\documentclass[12pt]{article}
\usepackage[utf8]{inputenc}
\usepackage{graphicx}
\usepackage{amsmath,amssymb,amsthm}
\usepackage{hhline}
\usepackage{wrapfig}  
\usepackage{appendix}

\renewcommand\le{\leqslant}
\renewcommand\ge{\geqslant}
\newcommand{\R}{\mathbb{R}}
\newtheorem{theorem}{Theorem}
\newtheorem*{theorem*}{Theorem}
\newtheorem*{theoremA}{Theorem A}
\newtheorem*{theoremB}{Theorem B}
\newtheorem{lemma}{Lemma}
\newtheorem{computation}{Computation}
\newtheorem*{corollary}{Corollary}
\newtheorem{statement}{Statement}
\theoremstyle{definition}
\newtheorem*{definition}{Definition}
\newcommand{\eps}{\varepsilon}

\DeclareMathOperator{\dil}{dil}
\DeclareMathOperator{\diam}{diam}
\DeclareMathOperator{\WD}{WD}
\DeclareMathOperator{\vol}{vol}

\title{Search of fractal space-filling curves with minimal dilation}
\author{Yuri Malykhin\thanks{Steklov Mathematical Institute. Corresponding
author: malykhin@mi-ras.ru}, Evgeny Shchepin\thanks{Steklov Mathematical
Institute}}

\begin{document}
\maketitle
\abstract{We introduce an algorithm for a search of extremal fractal curves in
large curve classes. It heavily uses SAT-solvers~--- heuristic algorithms that
find models for CNF boolean formulas. Our algorithm was implemented and applied
to the search of fractal surjective curves $\gamma\colon[0,1]\to[0,1]^d$ with
minimal dilation
$$
\sup_{t_1<t_2}\frac{\|\gamma(t_2)-\gamma(t_1)\|^d}{t_2-t_1}.
$$

We report new results of that search in the case of Euclidean norm.
We have found a new curve that we call ``YE'', a self-similar
(monofractal) plane curve of genus $5\times 5$ with dilation
$5\frac{43}{73}=5.5890\ldots$. 
In dimension $3$ we have found facet-gated bifractals (that we call
``Spring'') of genus $2\times2\times 2$ with dilation $<17$. In dimension $4$ we
obtained that there is a curve with dilation $<62$.

Some lower bounds on the dilation for wider classes of cubically
decomposable curves are proved.}

\section{Introduction}

A $d$-dimensional space-filling curve is a continuous map
$\gamma\colon[0,1]\to\mathbb R^d$ having image with non-empty interior. The
first examples of such curves for $d=2,3$ were constructed by G.~Peano in 1890.
In the computer era space-filling curves have found many practical applications,
see~\cite{B13}. The point is that these curves allow one to efficiently order
multidimensional lattices, which is important in some computations.

The constructions of space-filling curves are based on the idea of
self-similarity. Let us consider curves with cubic image $[0,1]^d$. Suppose that
the cube is divided on a number of sub-cubes and the curve on each
sub-cube (``fraction'') is similar to the whole curve. Such curve is a \textit{monofractal} and
the number of fractions is curve's \textit{genus}.
In this paper we consider also multifractal curves. We say that a tuple of
curves is a \textit{multifractal}, if each fraction of each curve is similar to
one of the curves (the precise definition is given below). This class is rich
enough to contain many interesting examples, and simple enough to make
computations with it.

An important property of Peano curves is their H\"older continuity.
It is easy to see that space-filling curves $\gamma\colon[0,1]\to\R^d$ are not
$\alpha$-H\"older for $\alpha>1/d$, but fractal curves are indeed
$(1/d)$-H\"older. See~\cite{S10} for more details. So we have a natural measure
of curve smoothness, or ``locality'',
$$
\sup_{t_1<t_2}\frac{\|\gamma(t_2)-\gamma(t_1)\|^d}{t_2-t_1},
$$
that we call the ``dilation''.
In this paper we treat the case of the Euclidean norm $\|x\|=\|x\|_2$.

We introduce an algorithm for a search of extremal fractal curves in
large curve classes. It relies on SAT-solvers~--- heuristic algorithms that
find models for CNF boolean formulas. We used our algorithm to find some new
curves with small dilation.

The monofractal Peano curve of genus $2\times 2$ is unique (up to isometry).
This is the most popular plane--filling curve~--- the Hilbert curve. It has
dilation $6$, as proved by K. Bauman~\cite{B06}.
Peano monofractals of genus $3\times 3$ were studied in detail in the papers of
E.V.~Shchepin and K.~Bauman~\cite{SB08,B11}. It was proven that the minimal
dilation equals $5\frac23$ and the unique minimal diagonal curve was
found. The same curve was found independently by H.~Haverkort and
F.~Walderveen~\cite{HW10} and named ``Meurthe''.
On the other hand, E.V.~Shchepin~\cite{S04} and K.~Bauman~\cite{B14} proved that any
Peano monofractal has the dilation at least $5$.
The best known plane Peano multifractal is the $\beta\Omega$ curve, which has
dilation $5$.

One of our results is a new curve that we call ``YE'', a 
monofractal plane curve of genus $5\times 5$ with dilation
$5\frac{43}{73}=5.5890\ldots$, which is best-known among plane
monofractals. Moreover, the YE-curve is the unique minimal
curve among monofractals of genus $\le 6\times 6$.  We give a proof of
minimality, which relies both on computations and theoretical
results. This result (together with the algorithm) was announced in~\cite{MS20}.

In the 3D case much of the research was done by H.~Haverkort. We refer to his
papers~\cite{H11} and~\cite{H17}.

We consider classes of facet-gated multifractals, i.e. multifractals with
entrances and exits on facets. These classes are rigid enough to
allow an efficient search, and contain many curves with small dilation. In
dimension $3$ we have found facet-gated bifractals of
genus $2\times2\times 2$ with dilation less than $17$. We give a name ``Spring''
for this family of curves. In dimension $4$ we obtained that
there is a curve with dilation less than $62$. As we are aware, these 
new curves are the best known (in terms of the dilation) in both dimensions $3$
and $4$.

Some lower bounds on the dilation for wider classes of cubically
decomposable curves are proved.

Let us outline the structure of the paper. In the next section we give
neccessary definitions. In Section~\ref{sec_results} we state the results of the
paper and compare them with known results, see Table~\ref{tbl_dil}. In
Section~\ref{sec_algo} we describe our algorithm, give a computer-assisted proof
of the minimality of the YE-curve and provide statistics of runs of our
algorithm. Section~\ref{sec_low} is devoted to the proofs of lower
bounds for the dilation. In the end of our paper we discuss possible directions of
further work.

\section{Definitions}

\paragraph{Dilation.}
Recall that a map $f$ between metric spaces $(T,\rho)$ and $(X,\sigma)$ is called
H\"older continuous with exponent $\alpha$ (or $\alpha$-H\"older, for
short) if for some $M$ the inequality holds
$$
\sigma(f(t_1),f(t_2))\le M\rho(t_1,t_2)^\alpha,\quad\forall t_1,t_2\in T.
$$
The smallest such $M$ is called H\"older coeffient, or H\"older seminorm (in case of
normed $X$), it is denoted as $|f|_{H^\alpha(T,X)}$.

Let us consider the $\ell_p^d$-norms in $\R^d$: 
$\|x\|_p=(|x_1|^p+\ldots+|x_d|^p)^{1/p}$, $1\le p<\infty$,
$\|x\|_\infty=\max|x_i|$. It is convenient to raise the corresponding
$(1/d)$-H\"older seminorm of $d$-dimensional space-filling curves to the $d$-th power.
\begin{definition}
    The $\ell_p$-dilation of a curve $\gamma\colon[a,b]\to\R^d$ 
    is the quantity
    \begin{equation}
        \label{WD}
        \WD_p(\gamma) := \sup_{a\le t_1<t_2\le b} \frac{\|\gamma(t_1)-\gamma(t_2)\|_p^d}{t_2-t_1}.
    \end{equation}
\end{definition}

It is obvious that
$$
\WD_p(\gamma) = |\gamma|^d_{H^{1/d}([a,b];\ell_p^d)}.
$$

Let us look at a slightly more geometrical approach to measure curve locality.
We say that a  curve $\gamma\colon[0,1]\to\R^d$ is
volume-parametrized, if for any $t_1<t_2$ we have
$\vol\gamma([t_1,t_2])=t_2-t_1$.
Define the $\ell_p$-dilation of a set $S\subset\R^d$ as
\begin{equation}
    \label{dilation}
\dil_p(S):=\frac{\diam_p(S)^d}{\vol(S)}.
\end{equation}
It is easy to see that for a curve with volume parametrization we
have
$$
\WD_p(\gamma) = \sup_{t_1<t_2} \dil_p(\gamma([t_1,t_2])).
$$
The last quantity depends only on the scanning order, not on the parametrization
of the curve.

The quantity $\WD_p$ was proposed as a measure of curve quality in the
paper~\cite{GL96}.  The notation $\WD_p$ is not well-established. Gotsman and
Lindenbaum in~\cite{GL96} intoduced notation $L_p(\gamma)$; Bader~\cite{B13}
writes $C_p$ for H\"older coefficient and $\widetilde{C}_p$ for its $d$-th power
(so, $\widetilde{C}_p=\WD_p$). Haverkort and Walderveen~\cite{HW10} suggested
the notation $\mathrm{WL}_p$ (Worst-case Locality) for~(\ref{WD});
Haverkort~\cite{H11} added the notation $\WD_p$ (Worst-case Diameter ratio) for
the equivalent quantity with diameter instead of distance.  Later,
in~\cite{H17}, he called $\mathrm{WL}_p$ the $L_p$-dilation.  We decided to
stick to the ``dilation'' term and the $\WD_p$ notation (Worst-case Dilation).

In this paper we consider only the Euclidean case, i.e. the Euclidean dilation
$\WD_2$, but our algorithms apply to any $p\in[1,\infty]$.

\paragraph{Peano multifractals.}

We say that curves $\gamma_1\colon I_1\to Q_1$ and $\gamma_2\colon I_2\to Q_2$
are \textit{isometric}, if there exist isometries $\alpha\colon I_1\to I_2$ and
$\beta\colon Q_1\to Q_2$, such that $\beta\circ\gamma_1 = \gamma_2\circ\alpha$.
Curves are \textit{similar} with a coefficient $s$, if $s\gamma_1(t/s^d)$ is
isometric to $\gamma_2(t)$.

Fix two parameters: a dimension $d\in\mathbb N$ and a scale $s\in\mathbb N$,
$s\ge 2$. The
unit cube $[0,1]^d$ is divided into the regular grid of $g:=s^d$ so-called
$s$-cubes
$$
\left[\frac{j_1}{s},\frac{j_1+1}{s}\right]\times\cdots
\times\left[\frac{j_d}{s},\frac{j_d+1}{s}\right].
$$
Moreover, we divide $[0,1]$ into $g$ equal segments $I_k^g:=[k/g,(k+1)/g]$. Given
a curve $\gamma$, each of the $g$ subcurves $\left.\gamma\right|_{I_k^g}$ is called
a \textit{fraction} of a curve.

\begin{definition}
    A tuple of surjective curves $\gamma_1,\ldots,\gamma_m\colon [0,1]\to[0,1]^d$
    forms a \textit{Peano multifractal} of genus $g=s^d$ if for all
    $r=1,2,\ldots,m$ each of the $g$ fractions $\left.\gamma_r\right|_{I_k^g}$
    maps $I_k^g$ to some $s$-cube $Q$ and is similar to one of the $\gamma_l$
    with the coefficient $s$.
\end{definition}
The parameter $m$ is called \textit{the multiplicity} of a multifractal.
For $m=1$ we get a Peano monofractal, for $m=2$ a bifractal, etc. A
related but broader class of polyfractal curves was introduced in~\cite{S15} for
dimension $3$.

Components of Peano multifractals are volume-parametrized curves.
They are $(1/d)$-H\"older continuous and the maximum
in~(\ref{WD}) is attained. If a multifractal $(\gamma_1,\ldots,\gamma_m)$ is
irreducible, i.e. does not contain a non-trivial subset $(\gamma_i)_{i\in I}$ that is a
multifractal itself, then $\WD_p(\gamma_1)=\ldots=\WD_p(\gamma_m)$. For the
generic case, one may formally define the dilation of a multifractal as
$\max_j\WD_p(\gamma_j)$.

Let us give some more definitions. A Peano monofractal of genus $s^d$ defines
the order (of traversal) of $s$-cubes. The sequence of $s$-cubes in that order
is called a \textit{prototype}. Peano multifractal
$(\gamma_1,\ldots,\gamma_m)$ has $m$ prototypes, accordingly. We remark that a
prototype is a special case of a polycubic chain, see Section~\ref{subsect_poly}.

The \textit{entrance} of a curve $\gamma\colon[0,1]\to\R^d$ is the point $\gamma(0)$,
and the \textit{exit} is $\gamma(1)$. They are also called entrance and
exit \textit{gates}. 
A \textit{pointed prototype} is a prototype with specified entrance/exit points for
each cube.

Let us give the definition, following~\cite{H17}.
\begin{definition}
A Peano multifractal $(\gamma_1,\ldots,\gamma_m)$ is called \textit{facet-gated} if all
entrances $\gamma_1(0),\ldots,\gamma_f(0)$ and exits $\gamma_1(1),\ldots,\gamma_m(1)$ lie
    on the cube facets (i.e., on $(d-1)$-dimensional faces) and not on faces of lower dimension.
\end{definition}

Each fraction $\left.\gamma_r\right|_{I_k^g}$ of a multifractal
$(\gamma_1,\ldots,\gamma_m)$ of genus $g$ is similar to some of the $\gamma_l$,
$l=l_{r,k}$.
That similarity is defined by a pair of isometry of the cube
$[0,1]^d$ (space orientation) and isometry of $[0,1]$ (time orientation). We
will refer to this pair as a \textit{base orientation}. So, a
multifractal is completely defined by
$m$ prototypes, $mg$ base orientations and $mg$ numbers $l_{r,k}$.
Note that base orientations form a group isomorphic to
$\mathbb Z_2^d\times S_d\times \mathbb Z_2$. We remark that base orientations
are not completely defined by a multifractal, if it has symmetries.

Each curve in the multifractal of genus $g$ is divided into $g$ fractions; they
are called fractions of order $1$. If we divide each of the fractions into $g$
sub-fractions, we will obtain $g^2$ fractions of order $2$, etc.

Often we identify fractions (curves) with their images (cubes).

A \textit{junction} $F\to G$ of a curve is a pair of adjacent fractions of some order
$k$. We call $F\to G$ a \textit{facet junction} if $F$ and $G$ intersect by a
common facet . The derived junction $F'\to G'$ is
defined as the junction of $(k + 1)$-th order formed by the last
fraction of the subdivision of $F$ and the first fraction of the subdivision of
$G$.  The \textit{depth} of a curve
(see~\cite{S04}) is the maximal order $k$ that generates new junctions.

An important observation made by E.~Shchepin~\cite{S15} is that the maximum in~(\ref{WD})
is attained at a pair of points from some non-adjacent sub-fractions $F_1\subset
F$, $G_1\subset G$ of some junction $F\to G$ (i.e., a pair from $F\times
G\setminus F'\times G'$, where $F'\to G'$ is the derived junction), or from some non-adjacent
first-order fractions $F_1$, $G_1$.

\paragraph{Cubically decomposable curves.}

We will require one more definition, from~\cite{KS18}.
\begin{definition}
    Consider a surjective curve $\gamma\colon[0,1]\to[0,1]^d$ and a number
    $g=s^d$. We say that $\gamma$ is decomposed into $g$ fractions if 
    each of the $g$ fractions $\left.\gamma_m\right|_{I_k^g}$
    maps $I_k^g$ onto some $s$-cube $Q$.
    The curve is called \textit{cubically decomposable} if it can be decomposed
    into an arbitrarily large number of fractions.
\end{definition}

Cubically decomposable curves are volume-parametrized.  Obviously, Peano
multifractal curves are cubically decomposable.

It is proven in~\cite{KS18} that cubically decomposable curves
have a recursive structure: there is some minimal number $s_1$ such that curve is
decomposable into $s_1^d$ first-order fractions; then there is some $s_2$ such
that each of the first-order fractions is decomposed into $s_2^d$ second-order
fractions, etc. The definition of a junction is the same as for multifractals: it
is a pair of order-$k$ adjacent fractions, for some $k$.

\section{Results}
\label{sec_results}

\subsection{Results on the $\ell_2$-dilation of monofractals.}

One of the main results of this paper is a Peano monofractal with
$\WD_2<5\frac23$.

\begin{theorem}
    \label{th_YE}
    There exists a unique (up to isometry) plane Peano monofractal curve
    $\gamma\colon[0,1]\to[0,1]^2$ of genus $5\times5$ such that
    $\WD_2(\gamma)=5\frac{43}{73}$. Moreover, it minimizes $\WD_2$ over
    monofractals of genus less or equal $6\times 6$.
\end{theorem}

\begin{figure}[!h]
    \begin{center}
        \includegraphics[width=12cm]{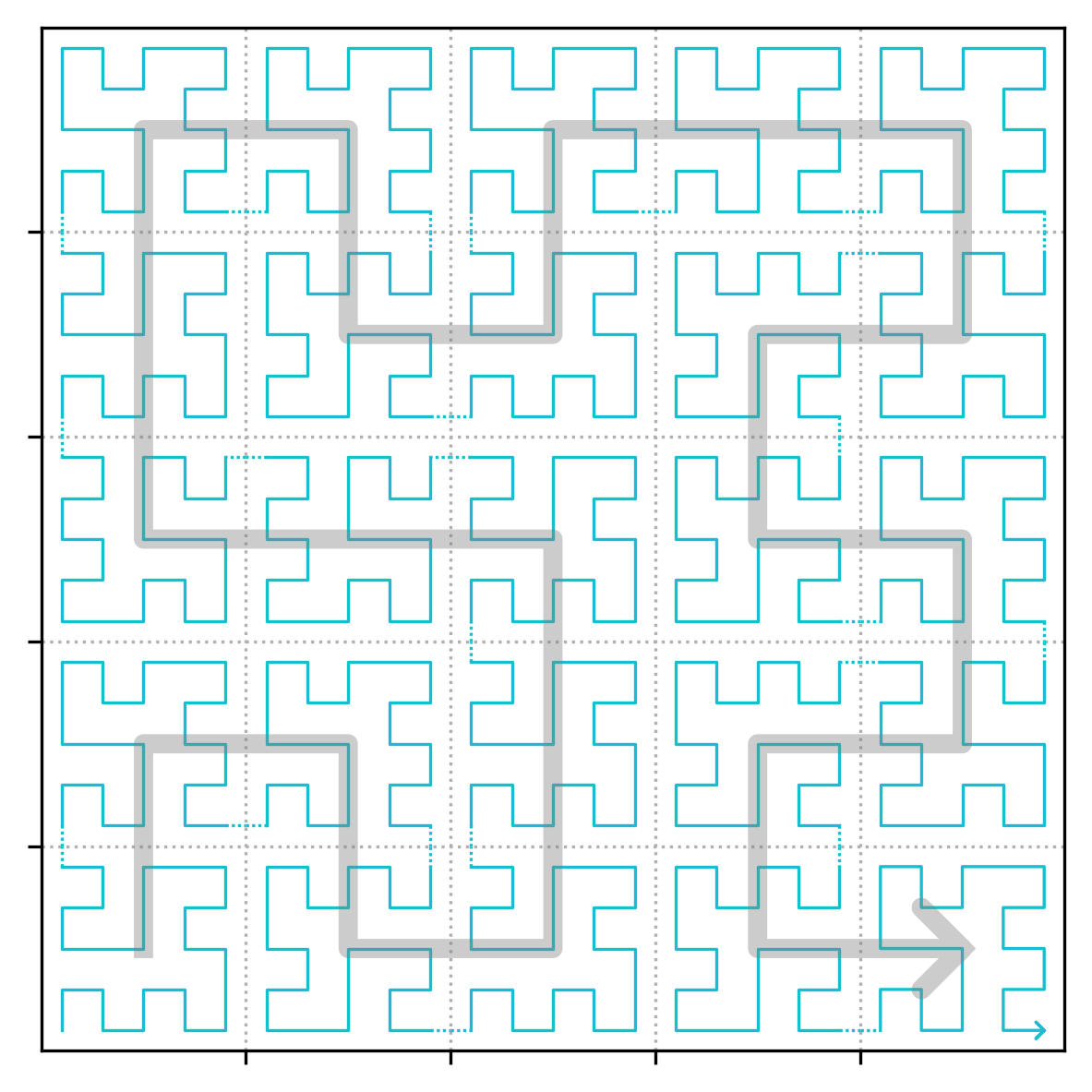}
    \end{center}
    \caption{YE minimal curve}
    \label{fig_YE}
\end{figure}

The fractal curve that we have found is shown in the Figure~\ref{fig_YE}.
More precisely, the figure shows the central broken lines of its prototype
and prototypes of its fractions; that uniquely defines it. The prototype of
this curve resembles stucked together letters Y and E, so we suggest the name
YE-shaped for curves of this prototype. It is proved that the dilation of
the curve shown in the figure is minimal among monofractals of genus $\le36$.
So we will call it minimal YE-shaped curve, or just YE-curve.

\subsection{Computations for facet-gated multifractals.}

Recall that the best known plane Peano multifractal, $\beta\Omega$ curve,
has entrances and exits on the sides of the square $[0,1]^2$ and not in
vertices. So, facet-gated curves are candidates for small $\ell_2$-dilation.

Facet-gated curves are facet-continuous, i.e., adjacent fractions
intersect by facets. But the gate condition is, of course, more
\textit{rigid} (restrictive) and allows us to make an effective search among
such curves.

We suppose that facet-gated curves minimize $\ell_2$-dilation in many
cases. This is true for the class of 2D bifractals of genus $2\times 2$ and for 3D
monofractals of genus $2\times 2\times 2$. Let us consider 3D bifractals.

\begin{definition}
We call any $3D$ facet-gated Peano bifractal of genus $2\times2\times2$ having
$\WD_2<17$ a \textit{Spring} curve.
\end{definition}

The following computation shows that Spring curves exist.

\begin{computation}
    \label{comp_spring}
    The minimal $\ell_2$-dilation of $3D$ facet-gated Peano bifractals of genus
    $2\times 2\times 2$ satisfies $16.9912 < \WD_2 < 16.9913$.
\end{computation}

Our hypothesis is that Spring curves are minimal among all
bifractals of genus $2^3$. Let us describe some properties of that curves.

Our calculations show that all Spring curves have dilation less than
$16.9913$; we suppose that their dilations are equal to
$1533\sqrt{6}/221$.
Moreover, all non-Spring facet-gated bifractals have $\WD_2 > 17.04$.
There are $9$ pointed prototypes containing Spring curves.

The entance/exit points and prototypes of Spring
curves are uniquely determined (up to isometry, of course). Spring bifractal
is a pair $(\gamma_1,\gamma_2)$; for short, we denote the first curve as (a) and
the second one as (b).
Entrance of (a) is $(1/3,1/3,0)$ and exit is $(2/3,0,1/3)$;
entrance of (b) is $(1/3,1/3,0)$ and exit: $(1,1/3,1/3)$.

To describe prototypes for (a) and (b), we use \textit{chain codes}
(see~\cite{KS18}). Fix a standard basis $\{\texttt{i,j,k}\}$ in $\R^3$ and let
$\texttt{I}=-\texttt{i}$, $\texttt{J}=-\texttt{j}$, $\texttt{K}=-\texttt{k}$.
\begin{figure}[h]
    \begin{center}
        \includegraphics[width=2cm]{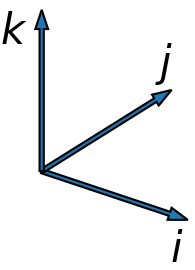}
    \end{center}
    \caption{Standard basis $\texttt{i,j,k}$.}
\end{figure}

Then, we can define a sequence of cubes $Q_1,\ldots,Q_{g}$ by sequence of vectors
$\delta_l$ such that $Q_{l+1}=Q_l+\delta_l$. For facet-continuous curves that
sequence may be represented as a word in the alphabet
$\{\texttt{i,j,k,I,J,K}\}$, called chain code.

In our case, the chain code for prototype (a) is \texttt{jikIJiK}
and for prototype (b) is \texttt{jkJijKJ}.

To give an example of a Spring curve, we have to specify base orientations and (a)/(b) letters
for each of the fractions.
A space orientation (an isometry of $[0,1]^3$) is completely defined by
its linear part, which is determined by images of basis vectors. Say,
$\texttt{KiJ}$ is the isometry with linear part that maps $\texttt{i}\to
\texttt{K}$, $\texttt{j}\to\texttt{i}$, $\texttt{k}\to\texttt{J}$. A time
orientation is either identity or time-reversal; we denote by $(\overline{a})$
time-reversed curve $(a)$ and similarly for $(b)$.

In our example we have the following fractions (see~Fig.~\ref{fig_spring}):

(a): \texttt{KIJ}($\overline{b}$),
\texttt{KIj}(a), \texttt{kji}(b), \texttt{Jki}($\overline{a}$),
\texttt{JkI}(b), \texttt{kIJ}(a), \texttt{KJi}(b), \texttt{JiK}(b).

(b): \texttt{KIJ}($\overline
b$), \texttt{ijK}($\overline{a}$), \texttt{Ikj}($\overline{a}$),
\texttt{KJI}($\overline{a}$), \texttt{kiJ}($\overline{a}$),
\texttt{Ijk}($\overline{a}$), \texttt{IjK}(a), \texttt{iKJ}(b).

\begin{figure}[h]
    \begin{center}
        \includegraphics[width=12cm]{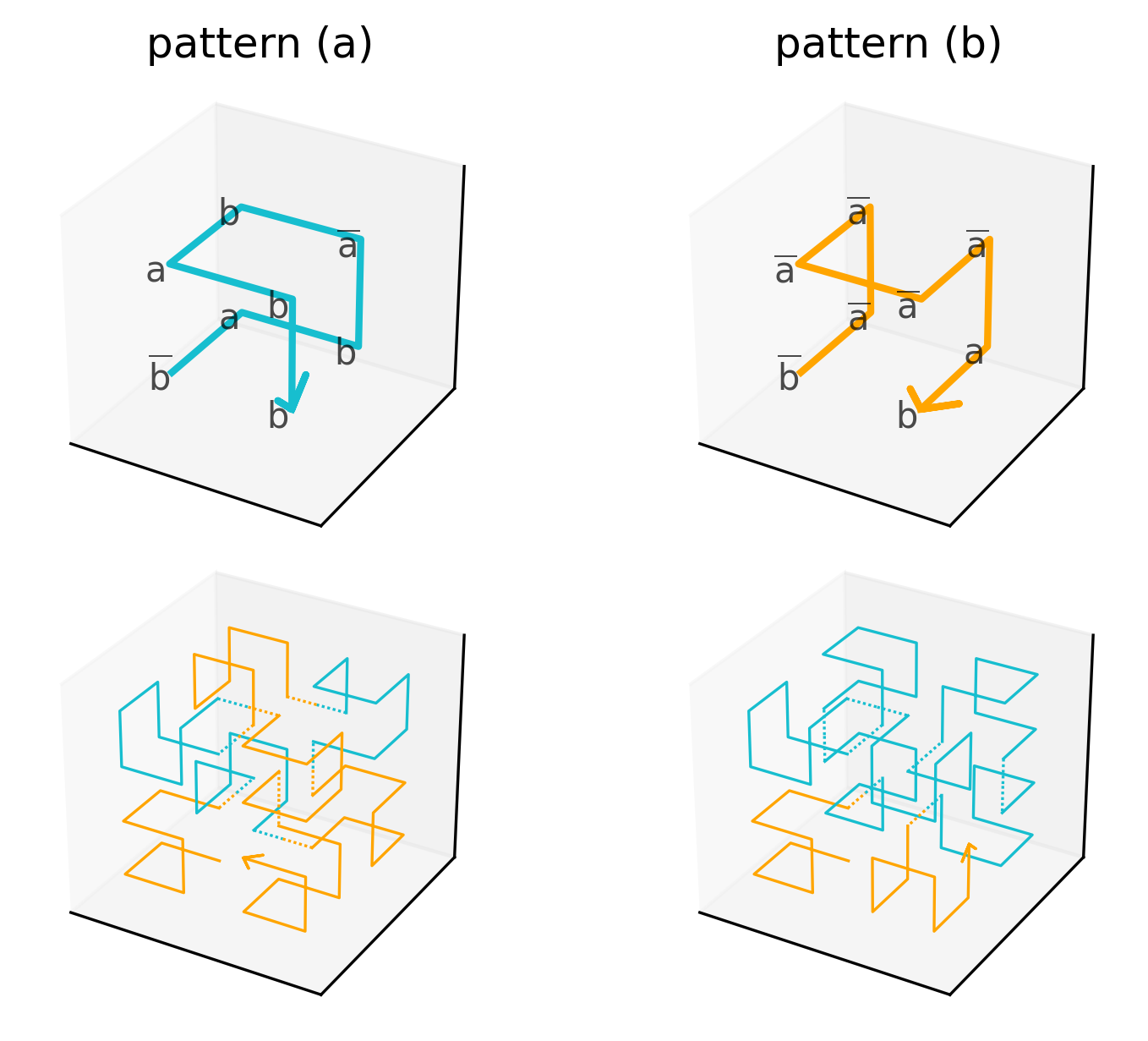}
    \end{center}
    \caption{Example of a Spring curve.}\label{fig_spring}
\end{figure}

\vspace{5pt}
We also have a result in the $4D$ case.

\begin{computation}
    The minimal $\ell_2$-dilation of $4D$ facet-gated Peano monofractals of genus
    $2^4$ safisties $61.935 < \WD_2 < 61.936$.
\end{computation}

Let us give an example of such curve. We use basis $\{\texttt{i,j,k,l}\}$.

Prototype: \texttt{kljKLkiKlkJLKlI}.
Base orientations:
\texttt{iKLJ}(a),
\texttt{kLij}(a),
\texttt{lJki}(a),
\texttt{jklI}(a),
\texttt{lKjI}($\overline{a}$),
\texttt{LKIj}(a),
\texttt{IkLj}($\overline{a}$),
\texttt{ikLj}(a),
\texttt{LKij}($\overline{a}$),
\texttt{lKji}(a),
\texttt{kjlI}(a),
\texttt{lJkI}($\overline{a}$),
\texttt{LkIJ}(a),
\texttt{LKji}($\overline{a}$),
\texttt{iljk}($\overline{a}$),
\texttt{jIlK}($\overline{a}$).
Entrance is $(0,1/3,1/3,1/3)$ and exit is $(1/3,0,1/3,2/3)$.
The dilation is at least $61\frac{29}{31}$ and we suppose that this is sharp.

\subsection{Lower bounds on the dilation.}

The proof of the YE-curve optimality relies on the following result.
\begin{theorem}
    \label{th_diag}
    If a plane Peano monofractal curve $\gamma$ has a diagonal junction, then
    $\WD_2(\gamma)\ge5\frac23$.
\end{theorem}

We also prove some lower bounds on the dilation for cubically decomposable
curves.

\begin{theorem}
    \label{th_poly2}
    There exists some $\eps>0$, such that for any plane cubically decomposable
    curve $\gamma\colon[0,1]\to[0,1]^2$ one has $\WD_2(\gamma)\ge 4+\eps$.
\end{theorem}

Note that the bound $\WD_2\ge4$ (and even stronger bound $\WD_\infty\ge 4$) was
proven in~\cite[Th.4]{HW10}.

\begin{theorem}
    \label{th_poly3}
    Let $\gamma\colon[0,1]\to[0,1]^3$ be a cubically decomposable curve. Then
    $\WD_2(\gamma)\ge 7\sqrt{14}/2>13.09$. If $\gamma$ has a non-facet junction, then
    $\WD_2(\gamma)\ge 17.17$.
\end{theorem}

From this Theorem and Computation~\ref{comp_spring} we obtain the following.

\begin{corollary}
    $3D$ Peano bifractals of genus $2\times 2\times 2$ with minimal
    $\ell_2$-dilation have only facet junctions.
\end{corollary}
We conjecture that minimal bifractals are facet-gated, but we do not have a
proof.

\begin{theorem}
    \label{th_poly4}
    Let $\gamma\colon[0,1]\to[0,1]^4$ be a cubically decomposable curve. Then
    $\WD_2(\gamma)\ge 42.25$. If $\gamma$ has a non-facet junction, then
    $\WD_2(\gamma)\ge 62.25$.
\end{theorem}

\begin{corollary}
    $4D$ Peano monofractals of genus $2^4$ with minimal $\ell_2$-dilation 
    have only facet junctions.
\end{corollary}

Finally, we prove a bound on the dilation for arbitrary surjective curves in the
4D case.
\begin{theorem}
    \label{th_4dsurj}
    For any surjective curve $\gamma\colon[0,1]\to[0,1]^4$ we have
    $\WD_2(\gamma)\ge 30$.
\end{theorem}

Table~\ref{tbl_dil} summarizes known results on the minimal $\ell_2$-dilation of
surjective curves~$\gamma\colon[0,1]\to[0,1]^d$. In the last column we normalize the
curve dilation by the dilation of the unit cube; this allows to compare
dilations across different dimensions and metrics.

\begin{table}
\begin{tabular}{|l|l|l|l|l|l|}
    \hline
    d & \textbf{genus} & \textbf{curve type} & \begin{tabular}{@{}l@{}}\textbf{minimal
    curve}\\\textbf{and references}\end{tabular} & $\WD_2$ &
    $\frac{\WD_2^{1/d}}{d^{1/2}}$ \\
    \hline
    \hline
    2 & $2\times 2$ & monofractal & Hilbert~\cite{B06} & 6 & 173\% \\
      & $3\times 3$ &  & Meurthe~\cite{SB08} & $5\frac23$ & 168\% \\
      & $\le 6\times 6$ &  & YE ($5\times 5$) & $5\frac{43}{73}$ & 167\% \\
    \hhline{|~|-----|}
      & any & monofractal &
        \begin{tabular}{@{}l@{}}\cite{S04,B14}\\best known: YE\end{tabular} &
        $\ge 5$ & $> 158\%$ \\
    \hhline{|~|-----|}
      & $2\times 2$ & bifractal & 
        \begin{tabular}{@{}l@{}}Wierum's $\beta\Omega$\\\cite{W02,HW10}\end{tabular} 
      & 5.000 & 158\% \\
    \hhline{|~|-----|}
      & \multicolumn{2}{l|}{any cubically decomposable} &
        \begin{tabular}{@{}l@{}}Theorem~\ref{th_poly2}\\best known: $\beta\Omega$\end{tabular} &  
        $>4$ & $>141\%$ \\
    \hhline{|~|-----|}
      & \multicolumn{2}{l|}{special triangle fractal} & Serpinski~\cite{NRS02} & 4 & 141\% \\ 
    \hhline{|~|-----|}
      & \multicolumn{2}{l|}{any surjective curve} &
        \begin{tabular}{@{}l@{}}\cite[Theorem 2]{MS21}\\best known: Serpinski\end{tabular} &
        $>3\frac58$ & $> 134\%$  \\
    \hline
    \hline
    3 & $2\times 2\times 2$ & monofractal & Haverkort's Beta~\cite{H17} & 18.566 & 153\% \\
    \hhline{|~|-----|}
      & $2\times 2\times 2$ &
        \begin{tabular}{@{}l@{}}facet-gated\\bifractal\end{tabular} &
            Spring (Computation 1) & 16.991 & 148\% \\
    \hhline{|~|-----|}
      & \multicolumn{2}{l|}{any cubically decomposable} &
        \begin{tabular}{@{}l@{}}Theorem~\ref{th_poly3}\\best known: Spring\end{tabular} &
            $> 13.095$ & $> 136\%$  \\
    \hhline{|~|-----|}
      & \multicolumn{2}{l|}{any surjective} &
        \begin{tabular}{@{}l@{}}\cite{NRS02}\\best known:  Spring\end{tabular} &
            $\ge 11.1$ & $> 128\%$ \\
    \hline
    \hline
    4 & $2^4$ &
        \begin{tabular}{@{}l@{}}facet-gated\\monofractal\end{tabular} &
            Computation 2 & 61.93 & 140\% \\
    \hhline{|~|-----|}
      & \multicolumn{2}{l|}{any cubically decomposable} &
        \begin{tabular}{@{}l@{}}Theorem~\ref{th_poly4}\\best known: Comp. 2\end{tabular} &
            $\ge 42.25$ & $> 127\%$  \\
    \hhline{|~|-----|}
      & \multicolumn{2}{l|}{any surjective} &
            \begin{tabular}{@{}l@{}}Theorem~\ref{th_4dsurj}\\best known: Comp.
            2\end{tabular} & $\ge 30$ & $> 117\%$ \\
    \hline
\end{tabular}

    \caption{Known results on $\ell_2$-dilation}\label{tbl_dil}
\end{table}

\section{Search algorithm and implementation}
\label{sec_algo}

\subsection{Algorithm}
The algorithms in this section work for arbitrary Peano multifractals, but for
simplicity we describe them for monofractals.

Recall that the prototype of a Peano monofractal of genus $s^d$ is the sequence
of $s$-cubes in the curve scanning order and the pointed prototype is th prototype
with specified entry/exit points for each cube.

To define a fractal curve with given pointed prototype, you must specify valid base
orientations for all fractions, i.e., such orientations that transform
entrance/exit points as required (note that time reversal swaps entrance and
exit). An important property of that structure is
that the choices for different fractions are independent.

Our computations are based on the algorithm for the estimation of the minimal dilation
within a class of curves with given pointed prototype.
 
Let us give an example to show the effeciency of such approach.
Consider curves of genus $5\times 5$ with entrance $(0,0)$ and exit $(1,1)$.
Fix some pointed prototype. It is clear that there are
four possible base orientations (two with time reversal and two without) in each
fraction.  This gives $2^{50}$ possible curves. Our algorithm allows us to
consider all that curves simultaneously.

\paragraph{Bisection of the curve dilation.}

It is advisable to consider the case of completely defined curve first.

We describe an algorithm that, given a monofractal curve $\gamma$ and thresholds
$\omega_-<\omega_+$,
either outputs pair of points on a curve providing $\WD_p(\gamma)\ge\omega_-$, or guarantees
that $\WD_p(\gamma)\le\omega_+$.

The dilation of the curve is equal to the maximum of the dilation for all
possible pairs of non-adjacent sub-fractions $F\subset F^\circ,G\subset G^\circ$
of curve junctions $F^\circ\to G^\circ$, or non-adjacent first-order fractions
$F,G$:
$$
\WD_p(F,G) := \max\left\{\frac{\|\gamma(t_1)-\gamma(t_2)\|_p^d}{|t_2-t_1|} \colon \gamma(t_1)\in
F,\;\gamma(t_2)\in G\right\}.
$$
There is a finite number of such pairs (becase the number of distinct junctions
is finite), and for each pair we can get upper $U$ and lower $L$ bounds for the
quantity $\WD_p(F,G)$:
$$
U = \frac{\max\{\|\gamma(t_1)-\gamma(t_2)\|_p\colon \gamma(t_1)\in F,\;\gamma(t_2)\in G\}}
{\min\{|t_2-t_1|\colon \gamma(t_1)\in F,\;\gamma(t_2)\in G\}},
$$
$$
L = \frac{\max\{\|\gamma(t_1)-\gamma(t_2)\|_p\colon \gamma(t_1)\in F,\;\gamma(t_2)\in G\}}
{\max\{|t_2-t_1|\colon \gamma(t_1)\in F,\;\gamma(t_2)\in G\}}.
$$

Three cases are possible: 
\begin{itemize}
    \item[(A)] If $L\ge\omega_-$, then $\WD_p(\gamma)\ge\omega_-$ and the work is done.  
    \item[(B)] If $U\le\omega_+$, then we can throw away (do not consider further) that pair.
    \item[(C)] If $L < \omega_- < \omega_+ < U$, then we subdivide both fractions and
        continue to evaluate pairs of sub-fractions.
\end{itemize}

Eventually, $U-L$ will be small enough to avoid the case (C). So, we either will
get (A), or we will throw away all pairs~--- it will mean that
$\WD_p(\gamma)\le\omega_+$.

\paragraph{Bisection of the minimal dilation for a pointed prototype.}

Now we describe an algorithm that, given a pointed prototype and thresholds
$\omega_-<\omega_+$, either guarantees that for all curves with that pointed prototype
$\WD_p(\gamma)\ge \omega_-$, or finds a curve with $\WD_p(\gamma)\le \omega_+$.

The difference from the previous algorithm is the following: in the case (C), when
subdividing fractions, we need to know the base orientations
on the corresponding fractions. Therefore, this step is
branching in accordance with all
possible base orientations. Next, at the case (A), when we arrive at a configuration with
$L\ge\omega_-$, we cannot state that $\WD_p(\gamma)\ge\omega_-$ for all curves in the class.
Instead, this only gives us a ``ban'' on all curves compatible with
orientations we selected in (C) that led to this configuration.

How to check that the forbidden configurations obtained at (A) cases exhaust all
curves of the class? We can describe the curves using boolean variables. Namely,
for each index of the fraction $i$ and for each valid base orientation $b$ we
introduce a boolean variable $x_{i,b}$, which is true if and only if the base
orientation at fraction $i$ is $b$. Our restrictions can be written as a CNF
formula over these boolean variables.  If this formula is unsatisfiable, it
means that all curves have the dilation at least $\omega_-$. Otherwise, if all
pairs are processed and the resulting formula has a model, then it will give us
a curve with $\WD_p\le\omega_+$.

The satisfiability problem of Boolean formulas (SAT-problem) is well studied
(see~\cite{BHM09}). There are many software libraries available to solve SAT tasks
(so-called SAT-solvers). Because the SAT belongs to the NP class, there is
no proven polynomial estimates for SAT-solvers.
The empirical fact we observed is that \textit{modern SAT-solvers work well on CNF
formulas, arising in the problem of finding optimal curves}.

\paragraph{Technical details.}
Let us give some more details of our algorithm that may have significant impact
on it's performance. When we subdivide pairs of fractions, we pick the
pair with the highest upper bound on the dilation. And when we do the subdivision,
we divide only one fraction in a pair (and keep pairs ``balanced'', i.e. the
depth of subdivision differs at most by $1$). If a pair falls into both
cases (A) and (B), i.e. $\omega_- \le L < U \le \omega_+$, then we choose (B),
i.e. the pair is discarded (not to enlarge CNF formula). 

The SAT-solver is used in the simplest way: every time we call the solve method,
we rebuild the CNF formula in it from scratch. To reduce computational costs of
this, we do this every $\lfloor1.3^k\rfloor$-th subdivision step.

To estimate the minimal dilation of a pointed prototype we use the bisection
algorithm many times. Suppose that we know some lower and upper bounds $(L,U)$ on the
minimal dilation. We put $\omega_- := \frac23L+\frac13U$ and $\omega_+ :=
\frac13L+\frac23U$. If there is a curve with $\WD\le\omega_+$, then the bounds
are updated as $L':=L$, $U':=\frac13L+\frac23U$. If every curve has
$\WD\ge\omega_-$, then $L'=\frac23L+\frac13U$, $U'=U$.

We estimate the minimal dilation for a sequence of pointed prototypes in 
several ``epochs''. In the first epoch we estimate the minimal dilation of each
pointed prototype with some relative error $\eps$ and keep only possibly
minimal prototypes. In the second epoch the relative error is taken smaller,
e.g. $\eps'=\eps/4$. And so on.

\subsection{Computer-assisted proof of Theorem~\ref{th_YE}}

In this section we will give a computer-assisted proof of Theorem 1 that
consists of two steps: finding the exact value of the dilation and proving that
it is minimal.

Our main tool that allows to find the exact value of the dilation is the following result.

\begin{theoremA}[\cite{SB08}, Theorem 7]
    If a plane Peano monofractal curve $\gamma$ of genus $g$ satisfies
    $\omega_\infty<\omega_2$, where $\omega_p:=\WD_p(\gamma)$,
    then the $\ell_2$-dilation of $\gamma$ is attained on a
    pair of vertices of fractions of order $k$, where
\begin{equation}
    \label{depth}
    k < \mathrm{depth}(\gamma)+3 +\log_g\frac{\omega_2^2}{4(\omega_2-\omega_\infty)}.
\end{equation}
\end{theoremA}

Recall that the depth of a curve is the largest number of the curve subdivision
that contains a junction that is not similar to junctions of less order. The minimal
YE-curve $\gamma_{\mathrm{YE}}$, see Fig.1, has depth $1$ because derivatives of
it’s junctions of first subdivision are similar to that junctions.  For our
curve we have $g=25$, $\omega_2=5\frac{43}{73}$ and $\omega_\infty=5\frac13$.
Hence from~(\ref{depth}) it follows that $k\le 5$ and we can calculate exact
value of $\WD_2$ using 5-th subdivision only.

The rigorous proof of the equality
$\omega_2=\WD_2(\gamma_{\mathrm{YE}})=5\frac{43}{73}=5.5890...$ goes in the
following way: first, we compute approximations for $\omega_2$ and
$\omega_\infty$ and prove that they differ from presumed
values $5\frac{43}{73}$ and $5\frac13$, say, less than $0.01$. These approximate
computations allow us to estimate the logarithm in~(\ref{depth}) and to find the
value of dilation using 5-th subdivision. 

Let us describe the computations that show that YE curve has minimal
$\ell_2$-dilation among all Peano monofractal curves of genus $\le 6\times 6$.
For brevity we will write $\WD=\WD_2$.
As we already mentioned, there is one $2\times2$ curve (Hilbert curve), it has
$\ell_2$-dilation $6$; the minimal $3\times 3$ curve has $\WD=5\frac23$.
Our computations show that minimum $\ell_2$-dilation of plane monofractal curves
of genus $4\times 4$ is greater than $5.9$.

It was proven in~\cite{B14} that plane Peano monofractal curves fall into three categories:
\begin{itemize}
    \item side: curves with entrance and exit in the vertices of the same side;
    \item diagonal: curves with entrance and exit in the opposite vertices;
    \item median: curves with entrance in some vertex and exit in the middle of
        the opposite side (or vice versa).
\end{itemize}

In the cases of genus $5\times5$ or $6\times 6$ we have processed only curves without
diagonal steps in prototype; other curves have diagonal junctions and by
Theorem~\ref{th_diag} they satisfy $\WD\ge5\frac23$. Our computations provide the following estimates:
median $6\times6$ curves have $\WD>6.5$; side $6\times6$ curves have $\WD>5.6$;
diagonal $6\times6$ curves do not exist; median $5\times 5$ curves have
$\WD>6.5$; diagonal $5\times 5$ curves have $\WD>5.9$.

The remaining case is side $5\times5$ curves.
Such curves with non-YE prototype have dilation at least $5.7$.
There are $2^{25}$ YE-shaped curves. To prove that our curve is minimal,
for each of the $25$ fractions of the YE-prototype we have fixed
the base orientation that does not match our curve. It turned out that for
curves with such constraints $\WD>5.6$.

\subsection{Software and running stats}

Our code is a Python package \texttt{peano} which can be found on the GitHub,~\cite{gitYM}.
We have used the library Glucose3~\cite{ALS13} and the \texttt{python-sat} (PySAT) wrapper~\cite{IMM18}
for SAT-solving. All calculations are based on rational arithmetic, with
\texttt{quicktions} implementation.

The \texttt{peano} package implements Peano multifractals of
arbitrary dimension, multiplicity and genus.  Let us list main capabilities of our library:
\begin{itemize}
    \item define Peano multifractals in terms of prototypes and base
        orientations;
    \item find entrance and exit points;
    \item find vertex moments: $\{t\colon \gamma(t)\in\{0,1\}^n\}$ and,
        more generally, first/last face moments for cube faces of
        arbitrary dimension;
    \item generate entrance/exit configurations given dimension, multiplicity
        and genus;
    \item effectively generate all entrance/exit configurations on facets;
    \item generate pointed prototypes given entrance/exit configuration;
    \item estimate dilation of a Peano multifractal for any coordinate-monotone
        norm (possibly using face moments);
    \item estimate minimal dilation for a sequence of pointed prototypes.
\end{itemize}

The code was thoroughly tested on a set of known Peano curves. The main source
of curves in the 3D case was the paper~\cite{H11} with many detailed
descriptions. We have used also~\cite{HW10} for a bunch of 2D curves
and~\cite{KS18} for the Tokarev curve.

In Table~\ref{tbl_run} we give some basic statistics for three search runs:
Computation 1, the search of a minimal curve among facet-gated 3D bifractals of
genus $2\times2\times2$; Computation 2, the search among facet-gated 4D monofractals of genus
$2\times2\times2\times2$; and the proof of minimality of YE curve, Theorem~\ref{th_YE}.
The commands to perform these runs are given in the README file of the package.
Let us give some comments on Table~\ref{tbl_run}.

The row ``Peano curves'' contains the numbers of all Peano curves with pointed
prototypes that were processed during the search. This number is rather small
for the Computation 1; it seems that the Spring curve can be found without SAT
solvers at all. The number of curves in the YE proof is huge; main contribution
comes from $11$ diagonal $5\times 5$ pointed prototypes (without diagonal steps), each
of them contains $4^{25}$ curves.

The size of a SAT problem for a CNF formula is determined by three parameters: number of variables,
number of clauses in the formula and total number of literals in clauses. We
show the SAT problem having the most literals.

\begin{table}
\begin{tabular}{|l|l|l|l|}
    \hline
    & Comp. 1 & Comp. 2 & YE proof \\
    \hline
    Gate configurations & 35 & 3 & 3 \\
    Pointed prototypes  & 909 & 1761 & 2763 \\
    Peano curves        & 122955 & $115\times10^6$ & $1.24\times10^{16}$ \\
    \hline
    Bisect steps        & 3301 & 8345 & 8472 \\
    Pairs of fractions & $10.4\times10^6$ & $143\times10^6$ & $204\times10^6$ \\
    Banned pairs  & 40840 & 815466 & $1.41\times10^6$ \\
    \hline
    SAT solve calls     & 35764 & 98164 & 47984 \\
    SAT largest problem: & & & \\
    \hspace{5pt} literals   & 19515 & 7825 & 480792 \\
    \hspace{5pt} clauses    & 7019  & 2405 & 82918 \\
    \hspace{5pt} variables  & 819   & 275 & 1713 \\
    \hline
    Running time & 4 min & 45 min & 1 h 20 min \\
    \hline
\end{tabular}
    \caption{Running stats}\label{tbl_run}
\end{table}

\section{Lower bounds}
\label{sec_low}

In this section we denote $|x|:=\|x\|_2$ for brevity.
Given two points $\gamma(t_1)$, $\gamma(t_2)$ on a curve, we define
the \textit{cube-to-linear ratio} as
$|\gamma(t_1)-\gamma(t_2)|^d/|t_2-t_1|$. The curve dilation equals
the supremum of cube-to-linear ratios over all pairs.

\subsection{Proof of Theorem~\ref{th_diag}}

\def\thetheorem{\ref{th_diag}}
\begin{theorem}
    If a plane Peano monofractal curve $\gamma$ has a diagonal junction, then
    $\WD_2(\gamma)\ge5\frac23$.
\end{theorem}

Suppose that consecutive fractions $F$ and $G$ of the curve $\gamma$ form a diagonal
junction. Let us put the transition point from $F$ to $G$ to the origin
$O=(0,0)$ and suppose that $G$ lies in the first, positive quadrant (right and above $O$)
and $F$ lies in the third, negative quadrant (left and below $O$). We also consider
fractions $F_1$ and $G_1$ of the next subdivision, such that $F\supset F_1\ni O$
and $G\supset G_1\ni O$. Moreover, let $G_2,G_3,G_4$ be consecutive fractions
that go after $G_1$ and let $F_4,F_3,F_2$ be consecutive fractions before $F_1$
(so, $F_2$ meets $F_1$). For convenience, we assume that $F_i$ and $G_i$ have
unit side length (and passed in the unit time). See Figure~\ref{fig_lem1} for an example.
We will treat several cases, illustrated in Figure~\ref{fig_cases}. Some of them
will require the following Lemma.
\begin{lemma}
    \label{lem_AB}
    Let $A$ be the left bottom vertex of $F_i$, and $B$ be the right upper
    vertex of $G_j$ and the conditions hold true:
    \begin{itemize}
        \item $i+j=6$,
        \item the distance from $A$ to the entrance of $F_i$ is at least
            $1$;
        \item the distance from $B$ to the exit of $G_j$ is at least
            $1$;
        \item $|A-B|\ge\sqrt{32}$.
    \end{itemize}
    Then $\WD_2(\gamma)\ge 5\frac23$.
\end{lemma}

\begin{figure}[!h]
    \begin{center}
    \includegraphics{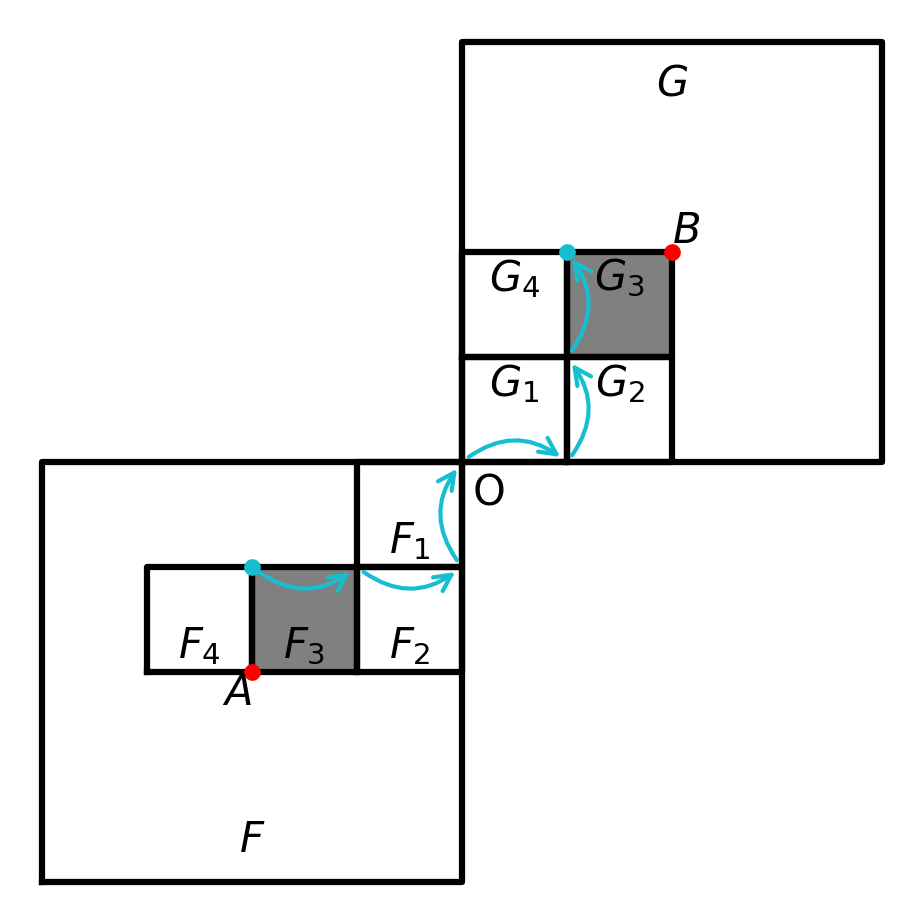}
    \end{center}
    \caption{Lemma~\ref{lem_AB}, $i=j=3$ (example).}
    \label{fig_lem1}
\end{figure}

\begin{proof}
    Suppose that $\WD_2(\gamma)\le 5\frac23$. Then the curve spends
    at least $\frac{3}{17}$ units of time to pass one unit of length.
    So the time interval between $A$ and $B$ is
    at most $6-\frac{6}{17} = \frac{6\cdot 16}{17}$ and the cube-to-linear ratio for $(A,B)$ is at least
    $$
    \frac{|A-B|^2}{t_B-t_A} \ge \frac{32\cdot 17}{6\cdot 16}=\frac{17}3=5\frac23.
    $$
\end{proof}

Let us return to the proof of the theorem.

\paragraph{Diagonal curves.}
If $\gamma$ is diagonal, then the farthest (from origin) point $A$ of $F_3\cup F_4$ has
$\|A-O\|_1 \ge 5$. The same holds for the farthest point $B$ of $G_3\cup G_4$ and hence
$\|A-B\|_1\ge 10$, so $|A-B|^2 \ge 5^2+5^2=50$. The time interval from $A$
to $B$ is at most $8$, so the cube-to-linear ratio for this pair is at least
$6\frac14$.

\paragraph{Median curves.}
If four fractions $G_1\ldots G_4$ do not lie in a $2\times 2$ square, then the
$\ell_1$-distance from the farthest point $B$ of $G_1\cup\ldots\cup G_4$ to $O$ is
at least $5$. And in any case the $\ell_1$-distance from the farthest point $A$
of $F_1\cup F_2\cup F_3$ to $O$ is at least $4$. Hence $\|A-B\|_1\ge 9$ and
$|A-B|^2\ge 4^2+5^2=41$. The time interval from $A$ to $B$ is at most $7$, so
the cube-to-linear ratio is $\ge\frac{41}{7}>5\frac23$. If fractions $F_1\ldots F_4$
do not lie in a $2\times 2$ square, the same arguments apply.

Finally, if both quads $F_1\ldots F_4$ and $G_1\ldots G_4$ lie in $2\times 2$ squares,
then the dilation bound follows from the Lemma~\ref{lem_AB} with $i=j=3$.

\begin{figure}
    \includegraphics[width=15cm]{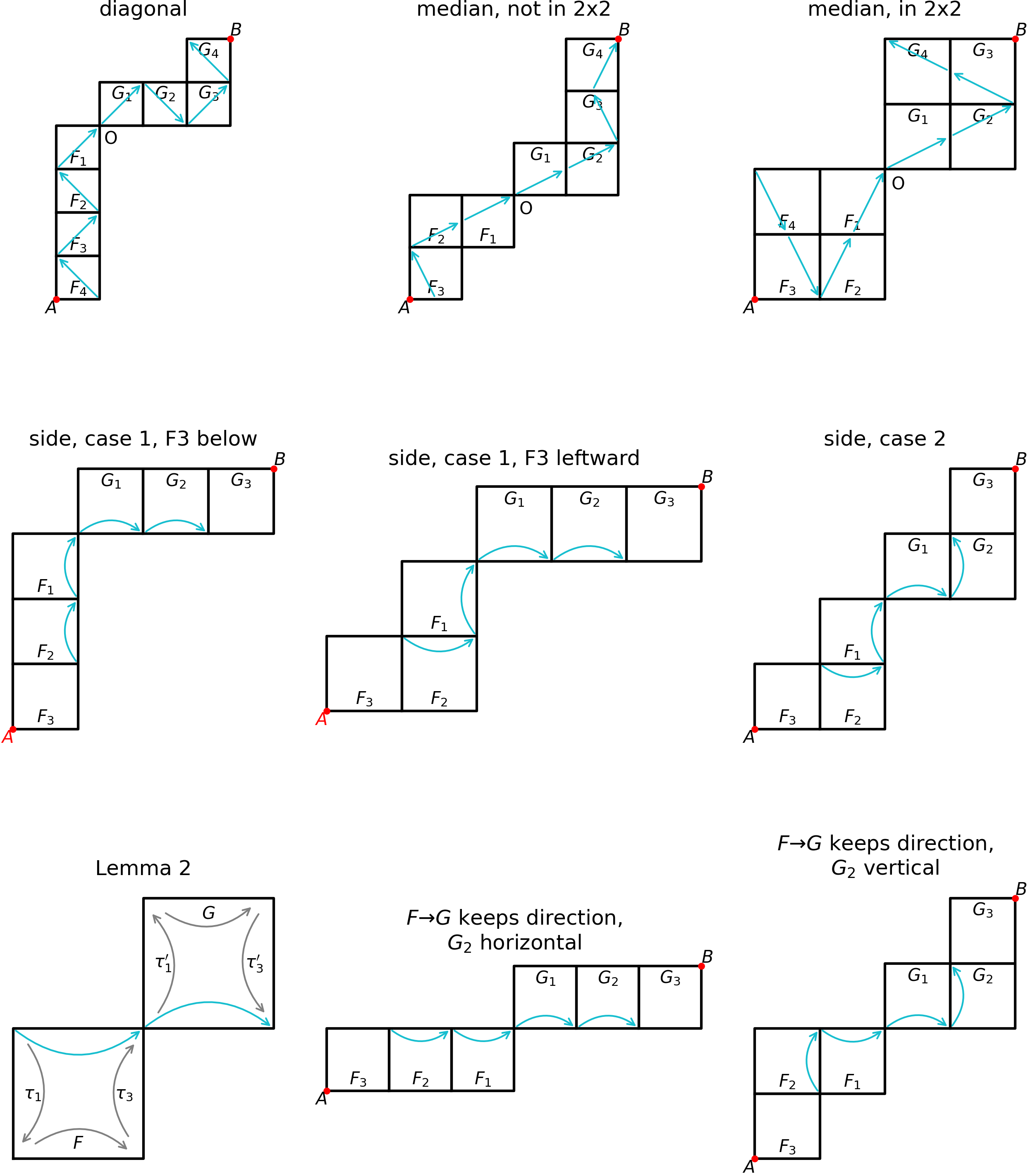}
    \caption{Cases in the proof of Theorem~\ref{th_diag}.}\label{fig_cases}
\end{figure}

\paragraph{Side curves.}
If the directions (from entrance to exit) of consecutive fractions coincide,
then we say that the corresponding junction \textit{keeps the direction}.

We will start with the case when $\gamma$ has no diagonal junctions that keep the
direction.

W.l.o.g. let $G_1$ has the horizontal direction, so $F_1$ is vertical. It follows
that $F_2$ is below $F_1$ and $G_2$ is to the right from $G_1$.

\textit{Case 1}. $G_2$ is horisontal. Then $G_3$ is to the right from $G_2$. Let
us consider fraction $F_3$: it is below $F_2$ or to the left from $F_2$; it
cannot be to the left from $F_1$ because in that case $F_3\to F_2$ will be a
diagonal junction that keeps the direction. In both cases we can apply
Lemma~\ref{lem_AB} for $F_3$ and $G_3$.

If $F_2$ is vertical, the reasoning is the same as in the Case 1.

\textit{Case 2}. $G_2$ is vertical, and $F_2$ is horizontal. Then $G_3$ must be
above $G_2$ and $F_3$~--- to the left from $F_2$. We apply Lemma~\ref{lem_AB}
for $F_3$ and $G_3$.

So, it remains to consider the case of a side curve with a diagonal junction keeping
direction. We may assume that the junction $F\to G$ keeps the direction.

Denote by $\tau_1(\gamma),\tau_2(\gamma)$ and $\tau_3(\gamma)$ the time
intervals from the curve entrance to the next visited vertex, then to the third
vertex and, finally, from it to the exit. We say that a curve is
\textit{accelerating} if $\tau_1(\gamma)>\tau_3(\gamma)$ and is
\textit{decelerating} if $\tau_1(\gamma)<\tau_3(\gamma)$. We will apply these
definitions to the fractions of $\gamma$.

\begin{lemma}
    \label{lem_accelerate}
    If a diagonal junction $F\to G$ keeps the direction and
    $\WD_2(\gamma)\le5\frac23$, then the first fraction $F$ is decelerating and
    the second, $G$, is accelerating. Moreover, we have
    $$
    \tau_1(F)=\tau_3(G)\le\frac5{17}<\frac{6}{17}\le\tau_3(F)=\tau_1(G).
    $$
\end{lemma}

\begin{proof}
    Let us consider first the case when vertices are visited along the edges, as
    in Figure~\ref{fig_cases}.
    Denote $\tau_1:=\tau_1(F)$, $\tau_3:=\tau_3(F)$, $\tau_1':=\tau_1(G)$,
    $\tau_3':=\tau_3(G)$.
    Condition $\WD_2(\gamma)\le 5\frac23$ gives us inequalities
    $\tau_3 + \tau_1' \ge \frac{4}{5\frac23} = \frac{12}{17}$
    and
    $2-\tau_1-\tau_3' \ge \frac{8}{5\frac23} = \frac{24}{17}$,
    so
    $$
    \tau_1+\tau_3' \le \frac{10}{17} < \frac{12}{17} \le \tau_3 + \tau_1'.
    $$
    This excludes case $\tau_1=\tau_1'$, $\tau_3=\tau_3'$. Since $\gamma$ is a
    monofractal, the only possible case is
    $$
    \tau_1=\tau_3' \le \frac{5}{17} < \frac{6}{17} \le \tau_3=\tau_1'.
    $$

    Now, suppose that the first and the second vertex visited in $F$
    lie diagonally opposite of each other (and therefore, so do the third and
    the fourth vertex, that is, $O$). Then we have
    $\tau_3+\tau_1'\ge\frac{24}{17}$ and $\tau_1\ge\frac{6}{17}$. If
    $\tau_1=\tau_1',\tau_3=\tau_3'$, then
    $\tau_3+\tau_1'+\tau_3'+\tau_1\ge\frac{48}{17}>2$, which is impossible. If
    $\tau_1=\tau_3',\tau_3=\tau_1'$, then $\tau_3\ge\frac{12}{17}$ and
    $\tau_1+\tau_3\ge\frac{18}{17}$, which is also impossible. So, this case
    cannot take place.
\end{proof}

From Lemma we derive that the junction $F\to G$ is centrally-symmetric. Since
our junction keeps the direction, we may assume that $F_1$ and $G_1$ are both
horizontal.

If $G_2$ is also horizontal, then $F_2$ has the same direction and the squared
diameter of $F_3\cup\ldots\cup G_3$ is at least $40$, which gives
$\WD_2(\gamma)\ge\frac{40}6$. 

If $G_2$ is vertical, then $\gamma$ runs the distance at least $2$ in time $1+\tau$,
where $\tau:=\tau_1(G_2)$. Since the whole $F\to G$ junction is
centrally-symmetric, $\gamma$ runs the distance $\ge 4$ in time $2+2\tau$. Hence
$\WD_2(\gamma)\ge \frac{16}{2+2\tau}$. If $\WD_2(\gamma)\le 5\frac23$, then it gives
$\tau\ge\frac{7}{17}$. Lemma~\ref{lem_accelerate} implies that $G_2$ is
accelerating. It also follows that junction $G_2\to G_3$ cannot be diagonal, so
$G_3$ is above $G_2$. Finally, application of Lemma~\ref{lem_AB} for $F_3$ and
$G_3$ finishes the proof.

\subsection{Polycubic chains}
\label{subsect_poly}

In this section we will prove theorems~\ref{th_poly2},~\ref{th_poly3} and~\ref{th_poly4}.

\begin{definition}
    \label{def_poly}
    A sequence $Q_1,\ldots,Q_n$ of distinct $d$-dimensional cubes in
    $d$-dimensional Euclid space if called a \textit{polycubic chain} if in some
    coordinates we have $Q_1=[0,1]^d$ and $Q_{i+1}=Q_i+\delta_i$, $1\le i\le
    n-1$, for some translation vectors $\delta_i\in\{-1,0,1\}^d$.
\end{definition}

The \textit{length} of a chain is the number of elements (cubes) in it. Elements of a
chain are called chain's \textit{fractions}. The union $Q_1\cup
Q_2\cup\ldots\cup Q_n$ of all fractions is
called the \textit{body} of a chain. A subsequence of cubes $Q_i,\ldots,Q_{i+k}$
is called a \textit{segment} of a chain (of course, it is a polycubic chain
itself). A \textit{junction} if a pair of adjacent fractions (i.e., a segment of
length $2$).

A chain is \textit{facet-continuous} if adjacent fractions intersect by facets.
Note that by definition polycubic chains are continuous, i.e., adjacent
fractions have non-empty intersection.

The \textit{dilation} of a polycubic chain is the maximal dilation of bodies of
its segments. (Recall the definition~(\ref{dilation}) of the set dilation.)

It is obvious that the sequence of fractions (of some order) of a cubically decomposable curve
is a polycubic chain and the dilation of the curve is at least the
dilation of this chain.

We say that fractions of a chain are collinear (coplanar), it the centers of them are
collinear (coplanar).

\paragraph{Dimension 2.}

We will use the following result.

\begin{theoremB}[\cite{MS21}, Theorem 1]
    If a plane surjective curve $\gamma\colon[0,1]\to[0,1]^2$ has entrance
    $\gamma(0)$ and exit $\gamma(1)$ on the opposite sides of $[0,1]^2$, then
    $\WD_2(\gamma)>4$.
\end{theoremB}

\begin{lemma}
    \label{lem_square}
    A 2D facet-continuous polycubic chain $S_1,\ldots,S_n$ with more than $4$
    fractions and square body has a collinear segment of length $3$.
\end{lemma}
\begin{proof}
    Consider the square $P=S_1\cup\ldots\cup S_n$.
    It has size $s\times s$ with $s^2=n\ge 9$, so $P$ has
    has a vertex that is not contained in $(S_1\cup S_2)\cup (S_{n-1}\cup S_n)$.
    This vertex is contained in some fraction $S_i$ with $3\le i\le n-2$. Denote by
    $S_j$ the only fraction that intersects $S_i$ by a vertex. If $j>i$ then one can
    take triple $S_{i-2},S_{i-1},S_i$ and if $j<i$ we take
    $S_i,S_{i+1},S_{i+2}$.
\end{proof}

\def\thetheorem{\ref{th_poly2}}
\begin{theorem}
    There exists some $\eps>0$, such that for any plane cubically decomposable
    curve $\gamma\colon[0,1]\to[0,1]^2$ one has $\WD_2(\gamma)\ge 4+\eps$.
\end{theorem}

\begin{proof}[Proof of Theorem~\ref{th_poly2}]
    If all junctions of a cubically decomposable curve $\gamma$ are facet, then one can take a decomposition into $n\ge 9$
    squares and apply Lemma~\ref{lem_square} to find a collinear segment
    $S_{i-1},S_i,S_{i+1}$. Then the middle fraction $S_i$ has its entrance and
    exit on the opposite sides on the square. We apply Theorem B to obtain that
    $\WD(\gamma)>4$.

    Suppose that there is a diagonal junction $F\to G$. Consider two last sub-fractions
    $F_2,F_1$ of $F$ and two first sub-fractions $G_1,G_2$ of $G$ (so, $F_1$
    meets $G_1$). Then
    $$
    \diam(F_2\cup F_1\cup G_1\cup G_2)\ge\min\{|(3,3)|,|(4,2)|\}=\sqrt{18},
    $$
    which gives $\WD\ge 18/4=4.5$.

    We have proved that $\WD(\gamma)>4$. The bound $\WD(\gamma)\ge4+\eps$ with
    some absolute $\eps>0$ follows from the compactness argument.
\end{proof}

\paragraph{Dimension 3.}

\begin{lemma}
    \label{lem_4empty}
    A 3D facet-continuous $4$-cubic chain with empty intersection of its fractions has dilation at
    least $7\sqrt{14}/2=13.0958\ldots$.
\end{lemma}

\begin{proof}
    Note that if the intersection of fractions of our chain $Q_1,\ldots,Q_4$
    is empty, then it is coplanar. In appropriate coordinates, the vector
    between two farthest points of $Q_1$ and $Q_4$ (the ``diameter vector'')
    equals $(3,2,1)$ or $(4,1,1)$ and 
    the diameter of the body $Q_1\cup\ldots Q_4$ at least $|(3,2,1)|=\sqrt{14}$, which
    gives the dilation bound $\ge 14\sqrt{14}/4$.
\end{proof}

\begin{lemma}
    \label{lem_5empty}
    If a polycubic chain with cubic body has more than $8$ fractions then it
    contains a $5$-cubic segment with empty intersection.
\end{lemma}

\begin{proof}
    Let $Q_1,\ldots,Q_n$ be a chain with a cubic body $P$. If $n>8$, then
    $P$ is a $s\times s\times s$ cube with $s\ge 3$. Any vertex of $P$ belongs
    to some fraction; let $Q_i$ be the fraction that contains the $4$-th vertex
    of P (in the chain traverse order). Then $7\le i\le n-8$.

    Let $k$ be the maximal number such that $Q_{i-1},Q_{i-2},\ldots,Q_{i-k}$
    intersect $Q_i$, and $l$ be the maximal number such that
    $Q_{i+1},\ldots,Q_{i+l}$ intersect $Q_i$. There are $7$ fractions
    intersecting $Q_i$, so $\min(k,l)\le 3$. Let, e.g., $l\le 3$. Then the chain
    $Q_i,\ldots,Q_{i+l+1}$ has at most $5$ fractions and empty intersection.
\end{proof}

\begin{statement}
    \label{stm_hyper3d}
    Any facet-continuous $3$-dimensional polycubic chain with cubic body and more
    than $8$ fractions has the dilation at least $7\sqrt{14}/2$.
\end{statement}

\begin{proof}
    If our chain contains a $4$-cubic segment with empty intersection, then 
    Lemma~\ref{lem_4empty} provides the required bound.

    Otherwise, consider a $5$-cubic chain $Q_1,\ldots,Q_5$ with empty intersection, given by
    Lemma~\ref{lem_5empty}. Fractions $Q_2,Q_3,Q_4$ must intersect by an edge.
    $Q_1$ and $Q_5$ must intersect that edge (otherwise a $4$-cubic chain with
    empty intersection appears), but they cannot contain that edge. Instead of that, they contain
    opposite endpoints of it. It follows that the diameter of the body
    $Q_1\cup\ldots\cup Q_5$ is at least $\sqrt{3^2+2^2+2^2}$ and the dilation is
    at least $17^{3/2}/5>14$.
\end{proof}

Now we are ready to prove lower bounds for the dilation of 3D cubically decomposable curves.

\def\thetheorem{\ref{th_poly3}}
\begin{theorem}
    Let $\gamma\colon[0,1]\to[0,1]^3$ be a cubically decomposable curve. Then
    $\WD_2(\gamma)\ge 7\sqrt{14}/2>13.09$. If $\gamma$ has a non-facet junction, then
    $\WD_2(\gamma)\ge 17.17$.
\end{theorem}

\begin{proof}[Proof of Theorem~\ref{th_poly3}]
    If all junctions of a cubically-decomposable curve $\gamma$ are facet, then we apply
    Statement~\ref{stm_hyper3d}.

    Suppose that $\gamma$ has a non-facet junction $F\to G$. If $F$ and $G$
    intersect by a vertex, then we have the dilation bound:
    $\WD(\gamma)\ge\diam(F\cup G)^3/2\ge 12^{3/2}/2>20$.
    So we may suppose that $\gamma$ has no such junctions.

    Let $F$ and $G$ intersect by an edge. Consider sub-fractions
    $F_3,F_2,F_1\subset F$ and $G_1,G_2,G_3\subset G$, such that $F_1$ meets
    $G_1$. 
    Let us choose coordinates $x,y,z$, such that the edge $F\cap G$ is
    parallel to the third (vertical) axis, the projections of $F$ and $G$
    onto horizontal plane $\{z=0\}$
    lie in the quadrants $\{x,y\le 0\}$ and $\{x,y\ge 0\}$, correspondingly. For definiteness,
    let $G_1=[0,1]^3$ and $F_1=G_1 + (-1,-1,0)$.

    Consider the fraction $G_2$.
    If $G_1$ and $G_2$ are on the same height, then
    $$
    \diam(F_2\cup F_1\cup G_1\cup G_2)\ge
    \min\{|(3,2,2)|,|(3,3,1)|,|(4,2,1)|\}=\sqrt{17},
    $$
    hence $\WD(\gamma)\ge 17\sqrt{17}/4>17.5$.
    If $G_1$ and $G_2$ intersect by an edge, then the diameter of $F_1\cup G_2$
    is already too large. So, we may assume that $G_2$ is above:
    $G_2=G_1+(0,0,1)$.

    Consider the fraction $G_3$. If the junction $G_2\to G_3$ is facet, then one
    can show that $\diam(F_1\cup G_3)\ge|(3,2,2)|$, which is sufficient for us.
    It remains one (up to isometry) ``difficult'' case, when $G_3$ and $G_2$
    intersect by an edge: $G_3 = G_1 + (1,0,0)$.

\begin{figure}[!h]
    \begin{center}
        \includegraphics[width=10cm]{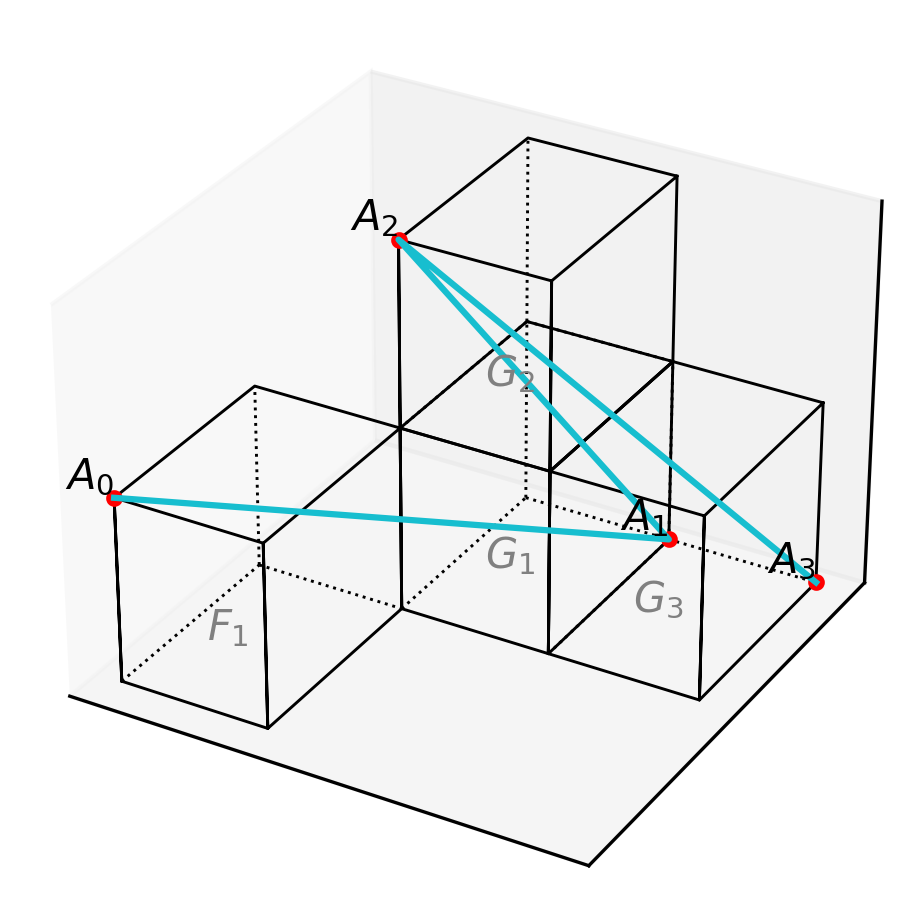}
    \end{center}
    \caption{Proof of Theorem~\ref{th_poly3}, ``difficult'' case.}
    \label{fig_difficult}
\end{figure}

    Consider the ``difficult'' case, illustrated in Figure~\ref{fig_difficult}. Take points $A_0=(-1,-1,1)\in F_1$,
    $A_1=(1,1,0)\in G_1$, $A_2=(0,0,2)\in G_2$, $A_3=(2,1,0)\in G_3$. There
    exists some moments $t_0<t_1<t_2<t_3$, such that $\gamma(t_i)=A_i$. Then by
    definition of $\WD$, we have
    $$
    \sum_{i=0}^2 |A_{i+1}-A_i|^3 \le \WD(\gamma)\sum_{i=0}^2 |t_{i+1}-t_i| \le
    4\WD(\gamma),
    $$
    $$
    \WD(\gamma) \ge \frac14\sum_{i=0}^2 |A_{i+1}-A_i|^3 =
    \frac14(27+6\sqrt{6}+27)>17.17.
    $$
\end{proof}

\paragraph{Dimension 4.}

\begin{lemma}
    \label{lem_7chain}
    If a $4D$ facet-continuous $7$-cubic chain $Q_1,\ldots,Q_7$ has a collinear subchain $Q_3,Q_4,Q_5$,
    then the dilation of $Q_1,\ldots,Q_7$ is more than $64$.
\end{lemma}

\begin{proof}
    Let us reduce the problem to a planar situation. Suppose that 5-chain
    $Q_2,\ldots,Q_6$ is not coplanar. Then the directions $Q_2Q_3$ and $Q_5Q_6$
    are orthogonal to each other and to the direction of $Q_3,Q_4,Q_5$. Then, in
    appropriate coordinates, the vector between two farthest points of $Q_2$ and
    $Q_6$ (the ``diameter vector'') equals $(3,2,2,1)$ and 
    the dilation is $18^2/5=64.8$.

    So, $Q_2,\ldots,Q_6$ is coplanar to some plane $\pi$. If one of the
    directions $Q_2Q_3$ or $Q_5Q_6$ is parallel to $Q_3,Q_4,Q_5$, then be have a
    collinear 4-chain. It has dilation $19^2/4>90$.

    If $Q_1,\ldots,Q_5$ or $Q_3,\ldots,Q_7$ are not coplanar to $\pi$, then the
    proof goes the same way as for $Q_2,\ldots,Q_7$. So, we may assume that the
    whole 7-chain is coplanar to $\pi$. The projection onto $\pi$ of one of the
    $5$-chains $Q_1\ldots Q_5$, $Q_2\ldots Q_6$, $Q_3\ldots Q_7$ does not fit
    into a $2\times 3$ rectangle and it gives a diameter vector $(3,3,1,1)$ or
    $(4,2,1,1)$ and the dilation at least $20^2/5=80$.
\end{proof}

\begin{statement}
    \label{stm_hyper4d}
    The dilation of 4D facet-continuous polycubic chain with cubic body and
    more than one fraction is at least $42.25$.
\end{statement}

\begin{proof}
    Denote our chain as $Q_1,\ldots,Q_n$. If there is a collinear 3-cubic chain
    $Q_i,Q_{i+1},Q_{i+2}$, $3\le i\le n-4$, then by Lemma~\ref{lem_7chain}
    the dilation at least $64$. Assume that such 3-cubic chains are not
    collinear. Let us prove that there is a 4-cubic chain $Q_i,Q_{i+1},Q_{i+2},Q_{i+3}$ that is
    not coplanar. It this is not the case, then the chain
    $Q_2,\ldots,Q_{n-1}$ is coplanar, because all adjacent 4-chains in it are
    coplanar to the same plane (defined by their three common fractions).
    Obviously, our chain cannot have cubic body in that case.

    If $Q_i,Q_{i+1},Q_{i+2},Q_{i+3}$ is not coplanar, then $Q_{i+3}$ is a
    translation of $Q_i$ by the vector $(1,1,1,0)$ (in appropriate coordinates).
    So, $\diam(Q_i\cup\ldots\cup Q_{i+3})\ge \sqrt{13}$ and we have the
    dilation bound $13^2/4=42.25$.
\end{proof}

\def\thetheorem{\ref{th_poly4}}
\begin{theorem}
    Let $\gamma\colon[0,1]\to[0,1]^4$ be a cubically decomposable curve. Then
    $\WD_2(\gamma)\ge 42.25$. If $\gamma$ has a non-facet junction, then
    $\WD_2(\gamma)\ge 62.25$.
\end{theorem}

\begin{proof}[Proof of Theorem~\ref{th_poly4}]
    If a cubically-decomposable curve $\gamma$ has only facet junctions, then we apply
    Statement~\ref{stm_hyper4d}.

    Suppose that $\gamma$ has a non-facet junction $F\to G$. If $F$ and $G$
    intersect by an edge, then we have the dilation bound:
    $$
    \WD(\gamma)\ge\frac12\diam(F\cup G)^4 = |(2,2,2,1)|^4/2 = 84.5
    $$
    and the proof is done. If $F$ and $G$ intersect by a vertex, there bound is even stronger.

    Let $F$ and $G$ intersect by a 2-dimensional face. Consider sub-fractions
    $F_3,F_2,F_1\subset F$ and $G_1,G_2,G_3\subset G$, such that $F_1$ meets
    $G_1$. 
    Let us choose coordinates $x,y,z,w$, such that $F\cap G$ is
    parallel to the third (vertical) axis $Oz$ and the fourth axis $Ow$, the projections of $F$ and $G$
    onto horizontal plane $\{z=w=0\}$
    lie in the quadrants $\{x,y\le 0\}$ and $\{x,y\ge 0\}$. For definiteness,
    let $G_1=[0,1]^4$ and $F_1=G_1 + (-1,-1,0,0)$.

    If the projections of $G_1$ and $G_2$ onto $Oxy$ do not coincide, then
    $\diam(F_1\cup G_2)\ge|(3,2,1,1)|=\sqrt{15}$, which gives the dilation bound
    $15^2/3=75$. So, $G_2$ may differ from $G_1$ only in the coordinates $z$ and $w$.
    If it differs in both coordinates, e.g., $G_2=G_1+(0,0,1,1)$, then we
    have the bound
    $\WD\ge \diam(F_1\cup G_2)^4/3\ge|(2,2,2,2)|^4/3=85\frac13$. So, we may
    assume that $G_2$ differs in $3$-rd coordinate: $G_2=G_1+(0,0,1,0)$.
    Analogously, $F_2$ differs from $F_1$ in either $3$-rd or $4$-th coordinate.
    Suppose, $F_2$ differs from $F_1$ in the $4$-th coordinate, say,
    $F_2=F_1+(0,0,0,1)$~--- then $\WD\ge\diam(F_2\cup
    G_2)^4/4\ge|(2,2,2,2)|^4/4=64$. So, $F_2=F_1+(0,0,\pm 1,0)$. The case ``-1''
    is impossible, since then $\diam(F_2\cup G_2)\ge|(3,2,2,1)|=\sqrt{18}$. 
    We obtain that $F_2=F_1+(0,0,1,0)$.

    Let us summarize the configuration that we have:
    $G_1=[0,1]^4$, $G_2=G_1+(0,0,1,0)$,
    $F_1=G_1+(-1,-1,0,0)$, $F_2=G_1+(-1,-1,1,0)$.

    Now we have to consider $G_3$ and the translation vector $\delta\colon
    G_3=G_2+\delta$. Note that $\delta_1,\delta_2\ge0$.
    
    The first case: $\delta_3=1$: we have
    $\diam(F_1\cup G_3)\ge|(3,2,2,1)|=\sqrt{18}$, which gives $\WD\ge
    18^2/4=81$.
    
    The second case: $\delta_3=0$ (i.e., $G_3$ and $G_2$ are on the same
    ``height''). As $\delta\ne\mathbf{0}$, we have $\delta_1=1$ or $\delta_2=1$
    or $\delta_4=\pm 1$. In any case, we have
    $$
    \diam(F_1\cup G_3)\ge \min\{|(3,2,2,1)|,|(2,2,2,2)|\}=4,\quad \WD\ge
    4^4/4=64.
    $$

    The last case: $\delta_3=-1$. If $\delta_1=1$ or $\delta_2=1$, then
    $\diam(F_2\cup G_3)\ge|(3,2,2,1)|$, $\WD\ge18^2/5=64.8$ and the proof is
    done. So we may assume that $\delta_1=\delta_2=0$. As $G_3\ne G_1$, we have $\delta_4\ne0$; so, w.l.o.g. we may
    assume that $\delta=(0,0,-1,1)$ and $G_3=G_1+(0,0,0,1)$.
    This ``difficult'' case
    needs special treatment, as in the proof of the previous theorem.

    Take the following points:
    $A_0=(-1,-1,1,0)\in F_1$, $A_1=(1,1,0,1)\in G_1$, $A_2=(0,0,2,0)\in G_2$,
    $A_3=(1,1,0,2)\in G_3$. There exist some moments $t_0<t_1<t_2<t_3$ such that
    $\gamma(t_i)=A_i$. By the definition of $\WD$, we have
    $$
    \sum_{i=0}^2 |A_{i+1}-A_i|^4 \le \WD(\gamma)\sum_{i=0}^2|t_{i+1}-t_i| \le
    4\WD(\gamma),
    $$
    so $\WD(\gamma)\ge\frac14(10^2+7^2+10^2)=62.25$.
\end{proof}

\subsection{Proof of Theorem~\ref{th_4dsurj}}

\def\thetheorem{\ref{th_4dsurj}}
\begin{theorem}
    For any surjective curve $\gamma\colon[0,1]\to[0,1]^4$ we have
    $\WD_2(\gamma)\ge 30$.
\end{theorem}

Vertices of the cube $[0,1]^4$ form the boolean cube $\{0,1\}^4$. In fact we will
obtain the dilation bound for any curve whose image contains $\{0,1\}^4$. We
say that vertices $v,v'\in\{0,1\}^4$ are \textit{3-antipodes} if they differ in
exactly $3$ coordinates, and \textit{4-antipodes}, if the differ in all $4$
coordinates.

\begin{lemma}
    Any set $S$ of $6$ elements of the boolean cube $\{0,1\}^4$ contains
    a pair of $4$-antipodes or two pairs of $3$-antipodes.
\end{lemma}
\begin{proof}
    We may assume that there are no $4$-antipodes in $S$.

    Let us prove that there is at least one pair of $3$-antipodes. 
    We may assume that $v=(0,0,0,0)\in S$. If $S$ has no $3$-antipodes, then all
    points in $S$ have at most two non-zero coordinates. So, $S$ contains a
    point, say, $w=(1,1,0,0)$ with two non-zero coordinates (because there
    are only $4$ vertices with one coordinate). But $w$ has $3$-antipodes
    $(0,0,1,0)$ and $(0,0,0,1)$, so that points can't lie in $S$. Then $S$
    contains one more point $u$ with two non-zero coordinates; as
    $u\ne(0,0,1,1)$, we may assume $u=(0,1,1,0)$. Then vertices $(1,0,0,0)$ and
    $(0,0,0,1)$ are forbidden, so $S$ must contain at least $4$ points with two
    non-zero coordinates. These points contain a pair of $4$-antipodes, that
    contradicts our assumption.

    So, we have a pair of $3$-antipodes, say, $v=(0,0,0,0)\in S$ and
    $v'=(1,1,1,0)\in S$. If there are no other $3$-antipodes, then all vertices
    in $S$ have at most two non-zero coordinates. But if the last coordinate of
    such point is non-zero, then this point is $3$-antipode to $v'$. So, $S$
    lies in the $3$-dimensional sub-cube $\{(x_1,x_2,x_3,x_4)\colon x_4=0\}$.
    And in the $3$-dimensional case the statement of Lemma is obvious.
\end{proof}

Now we can finish the proof of Theorem~\ref{th_4dsurj}. Let
$A_0 := \gamma(t_0),\ldots,A_{15} := \gamma(t_{15})$ be the sequence of points of $\{0,1\}^4$ in the
traversal order. Note that $(t_5-t_0)+(t_{10}-t_5)+(t_{15}-t_{10})\le 1$, so one
of the summands is at most $1/3$. Hence we have a sub-sequence
$A_s,\ldots,A_{s+5}$ of six points with $t_{s+5}-t_s\le 1/3$.

We apply Lemma to points $A_s,\ldots,A_{s+5}$ and arrive to
three possible cases: 1) there is a pair of $4$-antipodes 2) there are three
points that give two pairs of $3$-antipodes, 3) there
are four points that give two pairs of $3$-antipodes.

Denote $\omega := \WD_2(\gamma)$. 
In the case 1) we have a pair $A_i$, $A_j$ of $4$-antipodes,
$$
16 = |A_i-A_j|^4 \le \omega(t_j-t_i) \le \omega/3,
$$
In the case 2) we have $A_i,A_j,A_k$, $i<j<k$ and two of three
pairwise distances equal $\sqrt{3}$, so
$$
10 \le |A_i-A_j|^4 + |A_j-A_k|^4 \le \omega(t_k-t_i) \le \omega/3.
$$
Consider the last case of four points $A_i,A_j,A_k,A_l$. If the
antipode $B$ of $A_i$ is one of $A_j$, $A_k$, then
$$
10 \le |A_i-B|^4 + |B-t_l|^4 \le \omega(t_l-t_i) \le \omega/3.
$$
If $A_i$ and $A_l$ are antipodes, then $A_j$ and $A_k$ are antipodes too,
$$
11 \le |A_i-A_j|^4 + |A_j-A_k|^4 + |A_k-A_l|^4 \le \omega/3.
$$
In any case $\omega\ge30$ and the Theorem is proven.

\section{Further work}

There remain a lot of problems on the minimal $\ell_p$-dilation of Peano
multifractals, both theoretical and computational. Let us discuss some of them.

\paragraph{Promising directions.}

Our library allows to estimate minimal dilation for Peano
multifractals of arbitrary dimension $d$, multiplicity $m$ and genus $g$. The
number of possible variants grows rapidly as we increase any of the parameters
$d$, $m$, $g$; this raises two difficulties.  When we make a brute force search over
all pointed prototypes in some class, the number of pointed prototypes
corresponds to the ``width'' of the search, it is CPU-bounded. Given a pointed
prototype, we use a SAT-solver to approximate minimum dilation and the number of
possible curves corresponds to the ``depth'' of the search, it is memory (RAM)
bounded.  For example, for edge 3D trifractals of genus $2\times2\times2$ there
are not much pointed prototyes, but they are very
``deep''.  Conversely, there are too many pointed prototypes for edge 2D
monofractals of genus $6\times 6$; we have succeeded in the search among them
only because we could throw away prototypes with diagonal steps (and these
prototypes are not so deep).

Still it is possible to use the library in many interesting
cases. Here is a list of configurations to start with:
\begin{itemize}
    \item plane $7\times 7$ monofractals ($\ell_2$-dilation);
    \item 3D-bifractals of genus $2^3$ ($\ell_1$-dilation);
    \item 3D facet-gated trifractals of genus $2^3$ ($\ell_2$-dilation)
    \item 3D-monofractals of genus $3^3$;
    \item 4D facet-gated monofractals of genus $2^4$ ($\ell_1$, $\ell_\infty$-dilation);
\end{itemize}

Some of them will require optimizations in the algorithm or additional
constraints (on prototypes, base orientations, junctions, gates, etc) to reduce
the number of curves.  A good example is the narrow class of facet-gated curves.
Another interesting constraint is the \textit{symmetry} of a curve.

SAT-solvers may be applied in broader optimization, e.g., in the classes of
curves with fixed prototype~--- just add the clauses for adjacent fractions that
ensure the continuity of a curve.

\paragraph{Other metrics.}

Our algorithm for minimal $\ell_p$-dilation estimation in fact works for any
norms and other curve functionals.  This allows, e.g., to investigate curves in
boxes instead of cubes: $\gamma\colon[0,1]\to
[a_1,b_1]\times\cdots\times[a_d,b_d]$; it is equivalent to the consideration of
usual Peano curves in metric $\|x\| := \|(b-a)x\|_p$.

Another interesting family of metrics is based on the bounding boxes of the
sets $\gamma([t_1,t_2])$, see~\cite{HW10}. The structure of our algorithm
allows to minimize them analogously.

\paragraph{Acknowledgement.}
The authors express their gratitude to the anonymous referee for his careful
work and valuable suggestions. 

Data availability statement. The datasets generated during and/or analysed
during the current study are available in the github repository~\cite{gitYM}.

\appendix
\appendixpage

\section{Complexity of the search}

Fix three main parameters: dimension $d$, multiplicity $m$ and genus $g$. 
Let us consider the complexity of search of minimal curves among
Peano multifractals $(\gamma_1,\ldots,\gamma_m)$. For
simplicity, we suppose that all entrance/exit pairs $(\gamma_i(0),\gamma_i(1))$
are equivalent (i.e., may be transformed into each other by a base orientation) to some
pair $e=(p_0,p_1)$; we also assume that $e$ is equivalent to
$\widetilde{e}=(p_1,p_0)$. How many such curves are there for fixed $e$?

Let $N$ be the number of pointed prototypes with entrance/exit pairs equivalent
to $e$. Let there be $K$ ways to make the next step from $p_1\in[0,1]^d$, i.e.
to choose a neighbour cube and a next exit in it. For
example, if $e$ is an edge of a square, then $K=6$~--- there are $3$
neighbour squares and $2$ possible exits in each of them. If we consider
only facet-continuous edge curves, then $K=4$. So, the pointed prototype is
defined by $m$ words of length $(g-1)$ in $K$-alphabet. Hence, the
\textit{pointed prototype complexity} is bounded above by $m(g-1)\log_2K$, so, $N\le
K^{m(g-1)}$. This is of course a very rough bound, but to keep things simple we
will use it to measure the complexity.

Denote by $H$ the subgroup of base orientations that keep $e$.
Given a pointed prototype, on each fraction there are $m|H|$ ways to choose an
appropriate base orientation and curve index.
For $e$ being an edge of a square we have $|H|=2$, because there is an
identity map that keeps $e$ and there is a base orientation that consists of an
appropriate reflection and time reversal, that also keeps $e$.  Hence, curves
with fixed pointed prototype are defined by equations written as a word of length $mg$ in alphabet of
$m|H|$ symbols and the \textit{equations complexity} equals
$mg(\log_2m+\log_2|H|)$. Note that as $H$ is a subgroup of
$\mathbb Z_2^d\times S_d\times \mathbb Z_2$, we have $|H|\le d!2^{d+1}$.

Total number of curves equals $N(m|H|)^{mg}\le K^{m(g-1)}\cdot(m|H|)^{mg}$ and
the overall complexity may be written as
$$
\mbox{complexity} = \underbrace{m(g-1)\log_2K}_{\mbox{prototype}} +
\underbrace{mg(\log_2m+\log_2|H|)}_{\mbox{equations}}\;\mbox{bits}.
$$
The prototype complexity and equations complexity may serve as two indicators of
the complexity of the curve class and it makes sense to treat them separately.
In our algorithm, we make a brute force search over all pointed prototypes,
hence the prototype complexity roughly corresponds to the ``width'' of the
search, which is CPU-bounded. Given a
pointed prototype, we use a SAT-solver to approximate minimum dilation in
memory, so the equations complexity corresponds to ``depth'' of the search,
which is memory (RAM) bounded.

Let us consider some configurations that we are capable to deal with.

Plain monofractal edge curves of genus $5\times 5$. We have $K=4$ for facet-continuous
curves, and $K=6$ for arbitrary curves. In both cases $|H|=2$. It
gives the complexity $48+25$ bits in the facet-continuous case, and
$24\log_26+25\approx 62+25$ in the general case.

Diagonal $5\times5$ monofractals have $K=2$ in the facet-continous case and
$K=3$ in the general case. In both cases $|H|=4$. So we obtain 
$24+50$ bits and $24\log_23+50\approx 38+50$ bits, correspondingly.

To prove Theorem~\ref{th_YE} we had to process facet-continuous $6\times6$ curves.
The edge curves have compexity $70+36$.
The diagonal curves would have large equations complexity~-- $72$;
luckily enough, they do not exist!

Let us compare this with some cases that are too complex to process.

Edge $6\times 6$ curves have prototype complexity $\approx 90$, this is
beyond our capabilities. (So we had to use theoretical bound,
Theorem~\ref{th_diag}.)

The situation is even worse in dimension $3$.
Consider $3D$ monofractals of genus $3\times 3\times 3$ with entrance end
exit on the same edge. We have $K=21$, because there are $7$ neighbour cubes and
$3$ possible next exits in each of them. That gives us
prototype complexity $m(g-1)\log_2K = 27\log_221 \approx 119$; this is
too large for the brute-force search.

Another example is $3D$ edge trifractals of genus $2\times 2\times 2$. We have
$|H|=4$, that gives equations complexity $mg\cdot\log_212 \approx 86$ and it is too large
for our algorithm.

\end{document}